\documentclass[11pt]{article} 
\usepackage{amssymb}
\usepackage{amsmath}
\usepackage{amsthm}
\newtheorem{thm}{Theorem}
\newtheorem{prop}{Proposition}

\newtheorem{cor}{Corollary}

\newtheorem{lem}{Lemma}

\newcommand{\la}{\lambda} 

\newcommand{\Bcal}{\mathcal{B}}

\newcommand{\Lcal}{\mathcal{L}}

\newcommand{\Cbb}{\mathbb{C}}

\newcommand{\Rbb}{\mathbb{R}}

\newcommand{\Zbb}{\mathbb{Z}}

\newcommand{\lp}{\left(}
\newcommand{\rp}{\right)}
\newcommand{\lc}{\left\{}

\newcommand{\lb}{\left[}
\newcommand{\rb}{\right]}

\begin{document}

\renewcommand{\thefootnote}{\fnsymbol{footnote}}
\begin{flushright} OU-HET 912 \end{flushright} 

\begin{center}
{\LARGE {\bf Asymptotic formulae of two divergent bilateral basic hypergeometric series}}  
\vskip1.5cm
{\large 
{Hironori MORI\footnote{Department of Physics, Graduate School of Science, Osaka University, Toyonaka, Osaka 560-0043, Japan, hiromori@het.phys.sci.osaka-u.ac.jp}
and
Takeshi MORITA\footnote{Center for Japanese Language and Culture, Osaka University, Minoh, Osaka 562-0022, Japan, t-morita@cr.math.sci.osaka-u.ac.jp} 
}}
\end{center}

\vskip0.5cm
\begin{abstract} 
We provide new formulae for the degenerations of the bilateral basic hypergeometric function ${}_1\psi_1 ( a; b; q, z )$ with using the $q$-Borel-Laplace transformation. These are thought of as the first step to construct connection formulae of $q$-difference equation for ${}_1\psi_1 ( a; b; q, z )$. Moreover, we show that our formulae have the $q \to 1 - 0$ limit.
\end{abstract}

\renewcommand{\thefootnote}{\arabic{footnote}}

\section{Introduction}
In this paper, we give two asymptotic formulae for the bilateral basic hypergeometric series,
\[{}_1\psi_1(0;b;q,x):=\sum_{n\in\mathbb{Z}}\frac{1}{(b;q)_n}x^n\]
and
\[{}_1\psi_0(a;-;q,x):=\sum_{n\in\mathbb{Z}}(a;q)_n\left\{(-1)^nq^{\frac{n(n-1)}{2}}\right\}^{-1}x^n\]
from the viewpoint of the connection problems on $q$-difference equations. Here, $( a; q )_n$ is the $q$-shifted factorial defined by
\begin{align*}
( a; q )_n := \lc
	\begin{aligned}
	& 1, && n = 0, \\ 
	& ( 1 - a ) ( 1 - a q ) \dots ( 1 - a q^{n - 1} ), && n \ge 1, \\ 
	& \lb ( 1 - a q^{-1} ) ( 1 - a q^{-2} ) \dots ( 1 - a q^n ) \rb^{-1}, && n \le - 1. 
	\end{aligned}
\right.
\end{align*}
Notice that the $q$-shifted factorial is the $q$-analogue of the shifted factorial
\begin{align*} 
( \alpha )_n
=
\alpha \{ \alpha + 1 \} \cdots \{ \alpha + ( n - 1 ) \}.
\end{align*}
Moreover, $( a; q )_\infty := \lim_{n \to \infty} ( a; q )_n$ and we use the shorthand notation
\begin{align*}
( a_1, a_2, \dots, a_m; q )_\infty := ( a_1; q )_\infty ( a_2; q )_\infty \dots ( a_m; q )_\infty.
\end{align*}

Connection problems on $q$-difference equations are originally studied by G.~D.~Birkhoff \cite{Birkhoff}. At first, we review the connection problems on second order $q$-difference equations of the form
\begin{equation}\label{laplace}
(a_0+b_0x)u(q^2x)+(a_1+b_1x)u(qx)+(a_2+b_2x)u(x)=0,
\end{equation}
where $a_0a_2b_0b_2\not=0$. Let $u_1(x)$, $u_2(x)$ be independent solutions of equation \eqref{laplace} around the origin and let $v_1(x)$, $v_2(x)$ be those around infinity. We take suitable analytic continuation of $u_1(x)$ and $u_2(x)$ \cite{GR}. Then we obtain the connection formulae in the following matrix form:
\[\begin{pmatrix}
u_1(x)\\
u_2(x)
\end{pmatrix}
=\begin{pmatrix}
C_{11}(x)&C_{12}(x)\\
C_{21}(x)&C_{22}(x)
\end{pmatrix}
\begin{pmatrix}
v_1(x)\\
v_2(x)
\end{pmatrix}.
\]
Here, functions $C_{jk}(x)$, $(j,k=1,2)$ are $q$-periodic and unique valued, namely, elliptic functions. 
G.~N.~Watson gave the first example of the connection formula. He showed a connection formula for the (unilateral) basic hypergeometric series
\begin{equation}\label{2p1}
{}_2\varphi_1(a,b;c;q,x):=\sum_{n\ge 0}\frac{(a;q)_n(b;q)_n}{(c;q)_n(q;q)_n}x^n
\end{equation}
The series \eqref{2p1} satisfies the second order $q$-difference equation
\begin{equation}\label{heine}
(c-abqx)u(q^2x)-\{c+q-(a+b)qx\}u(qx)+q(1-x)u(x)=0
\end{equation}
around the origin. Equation \eqref{heine} also has solutions around infinity as follows:
\[v_1(x)=\frac{(ax,q/ax;q)_\infty}{(x,q/x;q)_\infty}
{}_2\varphi_1\left(a,\frac{aq}{c};\frac{aq}{b};q,\frac{cq}{abx}\right),\quad 
v_2(x)=\frac{(bx,q/bx;q)_\infty}{(x,q/x;q)_\infty}
{}_2\varphi_1\left(b,\frac{bq}{c};\frac{bq}{a};q,\frac{cq}{abx}\right).\]
The connection formula by Watson \cite{W} is 
\begin{align*}u_1(x)=&\frac{(b,c/a;q)_\infty}{(c,b/a;q)_\infty}\frac{(-ax,-q/ax;q)_\infty}{(-x,-q/x;q)_\infty}
\frac{(x,q/x;q)_\infty}{(ax,q/ax;q)_\infty}v_1(x)\\
&+\frac{(a,c/b;q)_\infty}{(c,a/b;q)_\infty}\frac{(-bx,-q/bx;q)_\infty}{(-x,-q/x;q)_\infty}
\frac{(x,q/x;q)_\infty}{(bx,q/bx;q)_\infty}v_2(x).
\end{align*}
We remark that each connection coefficient is given by the elliptic function. 

In equation \eqref{laplace}, if we assume that $a_0a_2b_0b_2=0$, some power series which appear in formal solutions may be divergent. Therefore, we should take a suitable resummation of a divergent series. J.-P.~Ramis and C.~Zhang introduced the $q$-Borel-Laplace resummation method to study the connection problems. The $q$-Borel-Laplace transformation is given as follows:

\begin{enumerate}
\item We assume that  $f(x)=\sum_{n\ge 0}a_nx^n$ is a formal power series. The $q$-Borel transformation of the first kind $\mathcal{B}_q^+$ is given by
\begin{align}
\left(\mathcal{B}_q^+f\right)(\xi ):=\sum_{n\ge 0}a_nq^{\frac{n(n-1)}{2}}\xi^n.
\label{qBorel}
\end{align}
We denote $ \varphi_f (\xi )=\left(\mathcal{B}_q^+f\right)(\xi )$. 
If  $f(x)$ is a convergent series, then $\varphi_f(\xi )$ is an entire function.

\item We fix $\lambda \in \mathbb{C}^*\setminus q^{\mathbb{Z}}$. For any entire function $\varphi (\xi )$, the $q$-Laplace transformation of the first kind $\mathcal{L}_{q,\lambda}^+$ \cite{Z0} is given by
\begin{align}
\left(\mathcal{L}_{q, \lambda}^+\varphi\right)(x):=
\frac{1}{1-q}\int_0^{\lambda\infty}\frac{\varphi (\xi )}{\theta_q\left(\frac{\xi}{x}\right)}\frac{d_q\xi}{\xi}=\sum_{n\in\mathbb{Z}}\frac{\varphi (\lambda q^n)}{\theta_q\left(\frac{\lambda q^n}{x}\right)},
\label{qLaplace}
\end{align}
where 
\begin{equation*}
\int_0^{\lambda\infty} f(t)d_qt:=(1-q)\lambda \sum_{n\in\mathbb{Z}}f(\lambda q^n)q^n
\end{equation*}
is Jackson's $q$-integral on $(0, \lambda \infty)$ \cite{GR}.  

\end{enumerate}
Thanks to these resummation methods, Zhang gave a connection formula for a divergent series ${}_2\varphi_0(a,b;-;q,x)$ \cite{Z0}. Morita also gave connection formulae for a divergent series ${}_2\varphi_0(0,0;-;q,x)$ \cite{M2} and some unilateral divergent series \cite{M4} by these transformations. 

But when we consider the connection formulae for the \textit{bilateral} series, the connection problems are not so clear. Though L.~J.~Slater  gave a relation between the bilateral series ${}_r\psi_r$ in \cite{Slater}, other relations are not known well. The main aim of this paper is to give (connection) formulae between divergent bilateral series and convergent unilateral series by $q$-Borel-Laplace transformations in Section \ref{degAB}. In the last section, we also give the classical limit $q\to 1-0$ of our new formulae.

We are closing this section with commenting on relationship with physics. We proved summation formulae of ${}_1\psi_1$ in \cite{Mori:2016eaz} which are politic generalizations of the ones found from the study of physics called Abelian mirror symmetry in three-dimensional supersymmetric gauge theories on the unorientable manifold $\Rbb\mathbf{P}^2 \times S^1$ \cite{TMM, MMT}. We hope our new formulae here could open up the class of connection formulae and provide prominent insight into physics.

\section{Notation}\label{Notation}
In this section, we declare basic notation. In the following, we assume that $0 < |q| < 1$. The bilateral basic hypergeometric series with the base $q$ is defined by
\begin{align} \label{bbh} 
{}_r\psi_s ( a_1, \dots, a_r; b_1, \dots, b_s; q, z ) 
:=
\sum_{n \in \mathbb{Z}}
\frac{( a_1, \dots, a_r; q )_n}{( b_1, \dots, b_s; q )_n}
\left\{ ( - 1 )^n q^{\frac{n ( n - 1 )}{2}} \right\}^{s - r} z^n.
\end{align}
This series diverges for $z \not = 0$ if $s < r$ and converges for $|b_1 \dots b_s/a_1 \dots a_r| < |z| < 1$ if $r = s$ (refer to \cite{GR} for more details).
Further, the series \eqref{bbh} is the $q$-analogue of the bilateral hypergeometric function
\begin{align*}
{}_rH_s ( \alpha_1, \dots, \alpha_r; \beta_1, \dots, \beta_s; z )
:=
\sum_{n \in \mathbb{Z}}
\frac{( \alpha_1, \dots, \alpha_r )_n}{( \beta_1, \dots, \beta_s )_n}
z^n,
\end{align*}
which is the bilateral extension of the generalized hypergeometric function
\begin{align*}
{}_{r}F_{s} ( \alpha_1, \cdots, \alpha_r; \beta_1, \cdots, \beta_s; z )
:=
\sum_{n \ge 0}
\frac{( \alpha_1, \cdots, \alpha_r )_{n}}{( \beta_1, \cdots, \beta_s )_{n}}
\frac{z^{n}}{n!}.
\end{align*}
Provided $\Re ( \beta_1 + \dots + \beta_r - \alpha_1 - \dots - \alpha_r ) > 1$, D'Alembert's ratio test could verify that ${}_rH_r$ converges only for $|z| = 1$ \cite{Slater}. For later use, we would like to write down Ramanujan's summation formula given by S.~Ramanujan \cite{Ramanujan},
\begin{align}
{}_1\psi_1 ( a; b; q, z )
=
\sum_{n \in \mathbb{Z}}
\frac{( a; q )_n}{( b; q )_n}
z^n
=
\frac{( q, b/a, a z, q/a z; q )_\infty}{( b, q/a, z, b/a z; q )_\infty}
\label{rsum11}
\end{align}
with $|b/a| < |z| < |1|$. Ramanujan's summation formula is considered as the bilateral extension of the $q$-binomial theorem \cite{GR} 
\begin{align} \label{qbinomial} 
\sum_{n \ge 0}
\frac{( a; q )_n}{( q; q )_n}
z^n
=
\frac{( a z; q )_\infty}{( z; q )_\infty}
\end{align}
for $|z| < 1$. The $q$-binomial theorem was derived by Cauchy \cite{C1}, Heine \cite{H1}, and many mathematicians.

One of the $q$-exponential functions is defined as
\begin{align} 
E_{q} ( z )
&:=
{}_{0}\varphi_{0} ( -; -; q, - z)
=
\sum_{n \geq 0}
{1 \over ( q; q )_{n}}
( - 1 )^{n} q^{n ( n - 1 ) \over 2}
( - z )^{n},
\end{align}
which can be rewritten by the infinite product expression
\begin{align} %
E_{q} ( z ) = ( - z; q )_{\infty}
\label{eq2}
\end{align}
with $|z| < 1$. We note that the limit $q \to 1 - 0$ of this $q$-exponential is actually the standard exponential
\begin{align}
\lim_{q \to 1- 0}
E_{q} ( z ( 1 - q ) )
=
e^z.
\end{align}

The $q$-gamma function $\Gamma_q ( z )$ is given by
\begin{align} 
\Gamma_q ( z )
:=
\frac{( q; q )_\infty}{( q^z; q )_\infty}
( 1 - q )^{1 - z}.
\end{align}
The $q \to 1 - 0$ limit of $\Gamma_q ( z )$ reproduces the gamma function \cite{GR}
\begin{align} 
\lim_{q \to 1 - 0}
\Gamma_q ( z ) = \Gamma ( z ).
\label{limgamma}
\end{align}
The theta function of Jacobi with the base $q$ which we will use is given by
\begin{align} \label{thetaJ} 
\theta_q ( z )
:=
\sum_{n \in \mathbb{Z}}
q^{\frac{n ( n - 1 )}{2}}
z^n,
\qquad \forall z \in \mathbb{C}^*.
\end{align}
With this definition, Jacobi's triple product identity is expressed by
\begin{align} 
\theta_q ( z )
=
( q, - z, - q/z; q )_\infty.
\label{jtpi}
\end{align}
The theta function has the inversion formula
\begin{align} \label{inv} 
\theta_q ( z )
=
\theta_q \left( q/z \right),
\end{align}
and satisfies the $q$-difference equation
\begin{align} 
\theta_q ( z q^k )
=
z^{- k} q^{- \frac{k ( k - 1 )}{2}}
\theta_q ( z ).
\label{thetaperi}
\end{align}
In our study, the following proposition about the theta function \cite{Z1, Mbi} is useful to consider the $q \to 1 - 0$ limit of our formulae in Section \ref{Limit}.
\begin{prop} \label{limthetaq} 
For any $z \in \mathbb{C}^* ( - \pi < \arg z < \pi )$, we have
\begin{align} %
\lim_{q \to 1 - 0}
\frac{\theta_q ( q^\beta z )}{\theta_q ( q^\alpha z )}
=
z^{\alpha - \beta}.
\label{p1}
\end{align}
and 
\begin{align} %
\lim_{q\to 1-0}
\frac{\theta_q \left( \frac{q^\alpha z}{( 1 - q )} \right)}{\theta_q \left( \frac{q^\beta z}{( 1 - q )} \right)} ( 1 - q )^{\beta - \alpha}
=
z^{\beta - \alpha}.
\label{limt2}
\end{align}
\end{prop}
We also use the following limiting formula of the ratio of the $q$-shifted factorial \cite{GR}:
\begin{align}
\lim_{q \to 1 - 0}
\frac{( z q^\alpha; q )_\infty}{( z; q )_\infty}
=
( 1 - z )^{- \alpha}, \qquad |z| < 1.
\label{limbin}
\end{align}

The crucial ingredients for our new formulae are the $q$-Borel-Laplace transformation.
In the rest of the paper, we concentrate on the sequence of the action of the $q$-Borel-Laplace transformation on degenerations of the divergent series ${}_{1}\psi_{1}^{\text{deg}}$,
\begin{align*}
{}_{1}\psi_{1}^{\text{deg}} ( x )
\xrightarrow[]{\Bcal_{q}^{+}}
\psi ( \xi )
\xrightarrow[]{\Lcal_{q, \la}^{+}}
\widetilde{\psi}_{\la} ( x ).
\end{align*}
We remark that $\la$-dependence on $\Lcal_{q, \la}^{+} \circ \Bcal_{q}^{+} f$ vanishes, i.e., $\Lcal_{q, \la}^{+} \circ \Bcal_{q}^{+} f = f$ if $f ( x )$ is a convergent series.

\section{Degenerations of the bilateral basic hypergeometric series}\label{degAB}
In this section, we will provide new convergent series obtained by acting the $q$-Borel-Laplace transformation on degenerations of the bilateral basic hypergeometric series ${}_1\psi_1(a;b;q,x)$ which satisfies the first order $q$-difference equation
\begin{equation}\label{1psi1}
\left(\frac{b}{q}-ax\right)u(qx)+(x-1)u(x)=0.
\end{equation}
We consider two different degenerations of equation \eqref{1psi1}. 
\begin{enumerate}
\item Degeneration A

In the equation \eqref{1psi1}, if we take the limit $a\to 0$, we obtain the following equation:
\begin{equation}\label{degA}
\frac{b}{q}\tilde{u}(x)+(x-1)\tilde{u}(x)=0.
\end{equation}
The bilateral series solution is given by 
\begin{equation}\label{dAbs}\tilde{u}(x)={}_1\psi_1(0;b;q,x).
\end{equation}
We remark that the series is a divergent series around the origin. We can also find a unilateral series solution \textit{around the origin} as follows:
\begin{equation}\label{dAus}
\tilde{v}(x)=\frac{\theta_q(bx)}{\theta_q(qx)} {}_1\varphi_0(0;-;q,x).
\end{equation}

\item Degeneration B

In the equation \eqref{1psi1}, if we put $x\mapsto bx$ and take the limit $b\to \infty$, we obtain another equation as follows:
\begin{equation}\label{degB}
\left(\frac{1}{q}-ax\right)\hat{u}(qx)+x\hat{u}(x)=0.
\end{equation}
The bilateral series solution is
\begin{align}
\hat{u}(x)={}_1\psi_0(a;-;q,x).
\label{dB1}
\end{align}
We remark that the solution $\hat{u}(x)$ contains a divergent series around the origin. We also find the unilateral basic hypergeometric series solution \textrm{around infinity} is 
\begin{equation}\label{dBus}
\hat{v}(x)=\frac{\theta_q(ax)}{\theta_q(x)}{}_1\varphi_0\left(0;-;q, \frac{1}{ax}\right). 
\end{equation}
\end{enumerate}
In the following subsection, we consider asymptotic formulae for these two divergent series.

\subsection{Degeneration A}
We here deal with the divergent series \eqref{dAbs},
\begin{align}
{}_{1}\psi_{1} ( 0; b; q, x )
=
\sum_{n \in \Zbb}
{1 \over ( b; q )_{n}} x^n.
\end{align}
As preparation, let us show the following formula:
\begin{lem} \label{limRam} 
For any $x \in \Cbb^\ast$, we have
\begin{align} %
{}_{0}\psi_{1} ( -; b; q, x )
=
{( q; q )_{\infty} \over ( b; q )_{\infty}}
{\theta_{q} \lp - x \rp \over \theta_{q} \lp - {q x \over b} \rp}
\lp {q x \over b}; q \rp_{\infty}.
\end{align}
\end{lem}
\begin{proof} 
Firstly, we scale $x \to x/a$ in the bilateral basic hypergeometric function,
\begin{align*} %
{}_{1}\psi_{1} ( a; b; q, x/a )
&=
\sum_{n \in \Zbb}
{( a; q )_{n} \over ( b; q )_{n}} {x^n \over a^n} \notag \\ 
&=
\sum_{n \in \Zbb}
{1 \over ( b; q )_{n}}
{( 1 - a ) ( 1 - a q ) \cdots ( 1 - a q^n ) \over a^n}
x^n \notag \\ 
&=
\sum_{n \in \Zbb}
{1 \over ( b; q )_{n}}
\lp {1 \over a} - 1 \rp \lp {1 \over a} - q \rp \cdots \lp {1 \over a} - q^n \rp
x^n. 
\end{align*}
Then, taking the limit $a \to \infty$ gives
\begin{align*} %
\lim_{a \to \infty}
{}_{1}\psi_{1} ( a; b; q, x/a )
=
\sum_{n \in \Zbb}
{1 \over ( b; q )_{n}}
( - 1 )^n q^{n ( n - 1 ) \over 2}
x^n
=
{}_{0}\psi_{1} ( -; b; q, x ).
\end{align*}
On the other hand, we repeat the same process into Ramanujan's summation formula \eqref{rsum11},
\begin{align*} %
\lim_{a \to \infty}
{}_{1}\psi_{1} ( a; b; q, x/a )
=
\lim_{a \to \infty}
{( q, {b \over a}, x, {q \over x}; q )_{\infty}
\over
( b, {q \over a}, {x \over a}, {b \over x}; q )_{\infty}}
=
{( q, x, {q \over x}; q )_{\infty}
\over
( b, {b \over x}; q )_{\infty}}.
\end{align*}
Therefore,
\begin{align*} %
{}_{0}\psi_{1} ( -; b; q, x )
&=
{( q, x, {q \over x}; q )_{\infty}
\over
( b, {b \over x}; q )_{\infty}} \notag \\ 
&=
{( q; q )_{\infty} \over ( b; q )_{\infty}}
{( x, {q \over x}, q; q )_{\infty} \over ( {b \over x}, {q x \over b}, q; q )_{\infty}}
\lp {q x \over b}; q \rp_{\infty} \notag \\ 
&=
{( q; q )_{\infty} \over ( b; q )_{\infty}}
{\theta_{q} \lp - x \rp \over \theta_{q} \lp - {q x \over b} \rp}
\lp {q x \over b}; q \rp_{\infty}, 
\end{align*}
which we actually would like to show.
\end{proof}
The main purpose is to apply the $q$-Borel-Laplace transformation to the divergent series ${}_{1}\psi_{1} ( 0; b; q, x )$. As a result, we can find the following convergent series for the degeneration $a \to 0$ of ${}_{1}\psi_{1}$:
\begin{thm} \label{qLBA} 
For any $x \in \Cbb^\ast \setminus [ - \la; q ]$, we have
\begin{align*} %
\widetilde{\psi}_\la^{\mathrm{A}} ( x )
=
{( q; q )_{\infty} \over ( b; q )_{\infty}}
{\theta_{q} \lp \la \rp \theta_{q} \lp {\la q \over b x} \rp
\over
\theta_{q} \lp {q \la \over b} \rp \theta_{q} \lp {\la \over x} \rp}
{}_{1}\varphi_{0} ( 0; -; q, x),
\end{align*}
where $\widetilde{\psi}_\la^{\text{A}} ( x ) := \lp \Lcal_{q, \la}^{+} \circ \Bcal_{q}^{+} {}_{1}\psi_{1} \rp ( 0; b; q, x )$.
\end{thm}
\begin{proof} 
The $q$-Borel transformation \eqref{qBorel} of ${}_{1}\psi_{1} ( 0; b; q, x )$ provides
\begin{align*} %
\psi_{\text{A}} ( \xi )
&:=
\lp \Bcal_{q}^{+} {}_{1}\psi_{1} \rp ( 0; b; q, x ) \notag \\ 
&=
\sum_{n \in \Zbb}
{1 \over ( b; q )_{n}}
q^{n ( n - 1 ) \over 2}
\xi^n \notag \\ 
&=
{}_{0}\psi_{1} ( -; b; q, - \xi ). 
\end{align*}
Actually, ${}_{0}\psi_{1}$ has the degenerated version of the Ramanujan's summation formula shown in Lemma \ref{limRam}, that is,
\begin{align*} %
\psi_{\text{A}} ( \xi )
&=
{( q; q )_{\infty} \over ( b; q )_{\infty}}
{\theta_{q} \lp \xi \rp \over \theta_{q} \lp {q \xi \over b} \rp}
\lp - {q \xi \over b}; q \rp_{\infty} \notag \\ 
&=
{( q; q )_{\infty} \over ( b; q )_{\infty}}
{\theta_{q} \lp \xi \rp \over \theta_{q} \lp {q \xi \over b} \rp}
E_{q} \lp {q \xi \over b} \rp. 
\end{align*}
Then, we take the $q$-Laplace transformation \eqref{qLaplace} to $\psi_{\text{A}} ( \xi )$,
\begin{align*} %
\widetilde{\psi}_\la^{\text{A}} ( x )
&:=
\lp \Lcal_{q, \la}^{+} \psi_{\text{A}} \rp ( \xi ) \notag \\ 
&=
\sum_{n \in \Zbb}
{{}_{0}\psi_{1} ( -; b; q, - \la q^{n} ) \over \theta_{q} \lp {\la q^{n} \over x} \rp} \notag \\ 
&=
{( q; q )_{\infty} \over ( b; q )_{\infty}}
\sum_{n \in \Zbb}
{\theta_{q} \lp \la q^{n} \rp \over \theta_{q} \lp {q \la q^{n} \over b} \rp}
{1 \over \theta_{q} \lp {\la q^{n} \over x} \rp}
\sum_{m \geq 0}
{( - 1 )^{m} q^{m ( m - 1 ) \over 2} \over ( q; q )_{m}}
\lp - {q \la q^{n} \over b} \rp^{m} \notag \\ 
&=
{( q; q )_{\infty} \over ( b; q )_{\infty}}
\sum_{n \in \Zbb}
{\la^{- n} q^{- {n ( n - 1 ) \over 2}} \theta_{q} \lp \la \rp \over \lp {q \la \over b} \rp^{- n} q^{- {n ( n - 1 ) \over 2}} \theta_{q} \lp {q \la \over b} \rp}
{1 \over \lp {\la \over x} \rp^{- n} q^{- {n ( n - 1 ) \over 2}} \theta_{q} \lp {\la \over x} \rp}
\sum_{m \geq 0}
{( - 1 )^{m} q^{m ( m - 1 ) \over 2} \over ( q; q )_{m}}
\lp - {q \la q^{n} \over b} \rp^{m} \notag \\ 
&=
{( q; q )_{\infty} \over ( b; q )_{\infty}}
{\theta_{q} \lp \la \rp \over \theta_{q} \lp {q \la \over b} \rp \theta_{q} \lp {\la \over x} \rp}
\sum_{n \in \Zbb}
\sum_{m \geq 0}
\lp {\la q \over b x} \rp^{n + m} q^{{( n + m ) ( n + m - 1 ) \over 2}}
{1 \over ( q; q )_{m}}
x^{m} \notag \\ 
&=
{( q; q )_{\infty} \over ( b; q )_{\infty}}
{\theta_{q} \lp \la \rp \over \theta_{q} \lp {q \la \over b} \rp \theta_{q} \lp {\la \over x} \rp}
\sum_{N \in \Zbb}
\lp {\la q \over b x} \rp^{N} q^{{N ( N - 1 ) \over 2}}
\sum_{m \geq 0}
{1 \over ( q; q )_{m}}
x^{m} \notag \\ 
&=
{( q; q )_{\infty} \over ( b; q )_{\infty}}
{\theta_{q} \lp \la \rp \over \theta_{q} \lp {q \la \over b} \rp \theta_{q} \lp {\la \over x} \rp}
\theta_{q} \lp {\la q \over b x} \rp
{}_{1}\varphi_{0} ( 0; -; q, x). 
\end{align*}
The final expression is obliviously a convergent series with $x \in \Cbb^\ast \setminus [ - \la; q ]$ as expected.
\end{proof}

\begin{cor}By Theorem \ref{qLBA} and \eqref{dAus}, we obtain the following relation:
\[\widetilde{\psi}_\la^{\mathrm{A}} ( x )=
{( q; q )_{\infty} \over ( b; q )_{\infty}}\frac{\theta_q(\lambda )}{\theta_q\left(\frac{q\lambda }{b}\right)}
\frac{\theta_q\left(\frac{bx}{\lambda}\right)}{\theta_q\left(\frac{qx}{\lambda}\right)}\frac{\theta_q(qx)}{\theta_q(bx)}\tilde{v}(x)=:C_{\mathrm{A}}(x)\tilde{v}(x).
\]
We remark that the function $C_{\mathrm{A}}(x)$ is an elliptic function, namely, $q$-periodic and unique valued.
\end{cor}

\subsection{Degeneration B}
Let us turn to divergent series \eqref{dB1},
\begin{align} %
{}_{1}\psi_{0} ( a; -; q, x ) 
=
\sum_{n \in \Zbb}
( a; q )_{n}
( - 1 )^{n} q^{- {n ( n - 1) \over 2}}
x^n.
\end{align}
Applying the $q$-Borel-Laplace transformation to ${}_{1}\psi_{0} ( a; -; q, x )$ can bring us to the following conclusion:
\begin{thm} \label{qLBB} 
For any $x \in \Cbb^\ast \setminus [ - \la; q ]$, we have
\begin{align*} %
\widetilde{\psi}_\la^{\mathrm{B}} ( x )
=
{( q; q )_{\infty} \over ( {q \over a}; q )_{\infty}}
{\theta_{q} ( a \la )
\theta_{q} \lp {a q x \over \la} \rp
\over
\theta_{q} ( \la )
\theta_{q} \lp {q x \over \la} \rp}
{}_{1}\varphi_{0} \lp 0; -; q, {1 \over a x} \rp,
\end{align*}
where $\widetilde{\psi}_\la^{\mathrm{B}} ( x ) := \lp \Lcal_{q, \la}^{+} \circ \Bcal_{q}^{+} {}_{1}\psi_{0} \rp ( a; -; q, x )$.
\end{thm}
\begin{proof} 
We perform at first the $q$-Borel transformation \eqref{qBorel},
\begin{align*} %
\psi_{\text{B}} ( \xi )
&:=
\lp \Bcal_{q}^{+} {}_{1}\psi_{0} \rp ( a; -; q, x ) \notag \\ 
&=
\sum_{n \in \Zbb}
( a; q )_{n}
( - \xi )^n \notag \\ 
&=
{}_{1}\psi_{1} ( a; 0; q, - \xi ). 
\end{align*}
Then, it is clear that Ramanujan's summation formula directly works on $\psi_{\text{B}} ( \xi )$ as
\begin{align*}
\psi_{\text{B}} ( \xi )
&=
{( q, - a \xi, - {q \over a \xi}; q )_{\infty}
\over
( {q \over a}, - \xi; q )_{\infty}} \notag \\ 
&=
{( q; q )_{\infty} \over ( {q \over a}; q )_{\infty}}
{\theta_{q} ( a \xi ) \over \theta_{q} ( \xi )}
E_{q} \lp {q \over \xi} \rp. 
\end{align*}
Finally, the $q$-Laplace transformation \eqref{qLaplace} of $\psi_{\text{B}} ( \xi )$ results in
\begin{align*} %
\widetilde{\psi}_\la^{\text{B}} ( x )
&:=
\lp \Lcal_{q, \la}^{+} \psi_{\text{B}} ( \xi ) \rp ( \xi ) \notag \\ 
&=
\sum_{n \in \Zbb}
{{}_{1}\psi_{1} ( a; 0; q, - \la q^{n} ) \over \theta_{q} \lp {\la q^{n} \over x} \rp} \notag \\ 
&=
{( q; q )_{\infty} \over ( {q \over a}; q )_{\infty}}
\sum_{n \in \Zbb}
{\theta_{q} ( a \la q^{n} ) \over \theta_{q} ( \la q^{n} )}
{1 \over \theta_{q} \lp {\la q^{n} \over x} \rp}
\sum_{m \geq 0}
{( - 1 )^{m} q^{m ( m - 1 ) \over 2} \over ( q; q )_{m}}
\lp - {q \over \la q^{n}} \rp^{m} \notag \\ 
&=
{( q; q )_{\infty} \over ( {q \over a}; q )_{\infty}}
\sum_{n \in \Zbb}
{( a \la )^{- n} q^{- {n ( n - 1 ) \over 2}} \theta_{q} ( a \la ) \over \la^{- n} q^{- {n ( n - 1 ) \over 2}} \theta_{q} ( \la )}
{1 \over \lp {\la \over x} \rp^{- n} q^{- {n ( n - 1 ) \over 2}} \theta_{q} \lp {\la \over x} \rp}
\sum_{m \geq 0}
{( - 1 )^{m} q^{m ( m - 1 ) \over 2} \over ( q; q )_{m}}
\lp - {q \over \la q^{n}} \rp^{m} \notag \\ 
&=
{( q; q )_{\infty} \over ( {q \over a}; q )_{\infty}}
{\theta_{q} ( a \la ) \over \theta_{q} ( \la ) \theta_{q} \lp {\la \over x} \rp}
\sum_{n \in \Zbb}
\sum_{m \geq 0}
\lp {\la \over a x} \rp^{n - m}
q^{( n - m ) ( n - m - 1 ) \over 2}
{1 \over ( q; q )_{m}}
\lp {1 \over a x} \rp^{m} \notag \\ 
&=
{( q; q )_{\infty} \over ( {q \over a}; q )_{\infty}}
{\theta_{q} ( a \la ) \over \theta_{q} ( \la ) \theta_{q} \lp {\la \over x} \rp}
\theta_{q} \lp {\la \over a x} \rp
{}_{1}\varphi_{0} \lp 0; -; q, {1 \over a x} \rp \notag \\ 
&=
{( q; q )_{\infty} \over ( {q \over a}; q )_{\infty}}
{
\theta_{q} ( a \la )
\over
\theta_{q} ( \la )
\theta_{q} \lp {q x \over \la} \rp
}
\theta_{q} \lp {a q x \over \la} \rp
{}_{1}\varphi_{0} \lp 0; -; q, {1 \over a x} \rp, 
\end{align*}
where in the last line the inversion formula \eqref{inv} of the theta function is used. Therefore, we find the desired convergent series with $x \in \Cbb^\ast \setminus [ - \la; q ]$.
\end{proof}

\begin{cor}By Theorem \ref{qLBB} and \eqref{dBus}, we obtain the following relation
\[\widetilde{\psi}_\la^{\mathrm{B}} ( x )=
{( q; q )_{\infty} \over ( q/a; q )_{\infty}}\frac{\theta_q(a\lambda )}{\theta_q\left(\lambda \right)}
\frac{\theta_q\left(\frac{aqx}{\lambda}\right)}{\theta_q\left(\frac{qx}{\lambda}\right)}\frac{\theta_q(x)}{\theta_q(ax)}\hat{v}(x)=:C_{\mathrm{B}}(x)\hat{v}(x).
\]
We remark that the function $C_{\mathrm{B}}(x)$ is an elliptic function, namely, $q$-periodic and unique valued.
\end{cor}

\section{The $q \to 1 - 0$ limit of the degenerations}\label{Limit}
We consider the limit $q\to 1-0$ of our formulae obtained in the previous section.
Before showing those limiting formula, we also focus on the limit $q\to 1-0$ of the $q$-binomial theorem in the following \textit{linear sum form} \cite{TMM}:
\begin{equation}\label{lin}
{}_1\varphi_0(a;-;q,x)
={}_2\varphi_1(a,aq;q;q^2,x^2)+x\frac{(a,q^3;q^2)_\infty}{(aq^2,q;q^2)_\infty} {}_2\varphi_1(aq,aq^2;q^3;q^2,x^2).
\end{equation}
\begin{prop} 
We put $a=q^\alpha$ in \eqref{lin}. Then, for any $x\in\mathbb{C}^*$, we have 
\[{}_1F_0(\alpha ;-;x)
={}_2F_1\left(\frac{\alpha}{2},\frac{\alpha}{2}+\frac{1}{2};\frac{1}{2};x^2\right)+
\alpha x{}_2F_1\left(\frac{\alpha}{2}+\frac{1}{2}, \frac{\alpha}{2}+1;\frac{3}{2};x^2\right).\]
\end{prop}
We can straightforwardly derive this sum expression and also verify this directly by decomposing the series expansion of ${}_1F_0(\alpha ;-;x)$ into the summations over even and odd integers. In addition, we consider the limit $q\to 1-0$ of the special case $a=0$ of \eqref{lin}:
\begin{equation}\label{lin2}
{}_1\varphi_0(0;-;q,x)
={}_2\varphi_1(0,0;q;q^2,x^2)+x\frac{(q^3;q^2)_\infty}{(q;q^2)_\infty}{}_2\varphi_1(0,0;q^3;q^2,x^2).
\end{equation}
\begin{prop} 
For any $x\in\mathbb{C}^*$, we have
\[_0F_0(-;-;2x)
={}_0F_1\left(-;\frac{1}{2};x^2\right)+ 2 x {}_0F_1\left(-;\frac{3}{2};x^2\right). \]
\end{prop}
\begin{proof} 
We put $x\mapsto (1-q^2)x$ in \eqref{lin2} ad then implement the limit $q \to 1 - 0$. The left-hand side of \eqref{lin2} converges to the exponential function,
\[\lim_{q\to 1-0}\sum_{n\ge 0}\frac{(1-q^2)^n}{(q;q)_n}x^n
=\lim_{q\to 1-0}\sum_{n\ge 0}\frac{(1-q)^n}{(q;q)_n}((1+q)x)^n
={}_0F_0(-;-;2x)=e^{2x}.\]
On the other hand, the limit $q \to 1 - 0$ of the right-hand side with $x\mapsto (1-q^2)x$ reduces to
\begin{align} %
&\lim_{q\to 1-0}\left\{\sum_{n\ge 0}\frac{(1-q^2)^{2n}}{(q;q^2)_n(q^2;q^2)_n}x^{2n}+
(1-q^2)x\frac{(q^3;q^2)_\infty}{(q;q^2)_\infty}\sum_{n\ge 0}\frac{(1-q^2)^{2n}}{(q^3;q^2)_n(q^2;q^2)_n}x^{2n}
\right\}\notag\\ 
&=\lim_{q\to 1-0}\sum_{n\ge 0}\frac{(1-q^2)^n(1-q^2)^n}{(q;q^2)_n(q^2;q^2)_n}x^{2n}\notag\\ 
&\hspace{1em} +
x\lim_{q\to 1-0}\frac{(q^2;q^2)_\infty}{((q^2)^{\frac{1}{2}};q^2)_\infty}(1-q^2)^{\frac{1}{2}}\times
\frac{((q^2)^{\frac{3}{2}};q^2)_\infty}{(q^2;q^2)_\infty}(1-q^2)^{\frac{1}{2}}
\sum_{n\ge 0}\frac{(1-q^2)^n(1-q^2)^n}{(q^3;q^2)_n(q^2;q^2)_n}x^{2n}\notag\\ 
&=\lim_{q\to 1-0}\sum_{n\ge 0}\frac{(1-q^2)^n(1-q^2)^n}{(q;q^2)_n(q^2;q^2)_n}x^{2n}
+x\lim_{q\to 1-0}\frac{\Gamma_{q^2}\left(\frac{1}{2}\right)}{\Gamma_{q^2}\left(\frac{3}{2}\right)}
\sum_{n\ge 0}\frac{(1-q^2)^n(1-q^2)^n}{(q^3;q^2)_n(q^2;q^2)_n}x^{2n}\notag\\ 
&=\sum_{n\ge 0}\frac{1}{\left(\frac{1}{2}\right)_n n!}x^{2n}
+x\frac{\Gamma\left(\frac{1}{2}\right)}{\Gamma\left(\frac{3}{2}\right)}\sum_{n\ge 0}\frac{1}{\left(\frac{3}{2}\right)_n n!}x^{2n}\notag\\ 
&={}_0F_1\left(-;\frac{1}{2};x^2\right)+2x{}_0F_1\left(-;\frac{3}{2};x^2\right). \notag 
\end{align}
Therefore, we obtain the conclusion. 
\end{proof}

\subsection{On the degeneration A}
Let us consider the limit $q\to 1-0$ of Theorem \ref{qLBA}.
We can find the well-defined limiting formula by scaling $x \mapsto ( 1 - q ) x$.
\begin{thm} 
For any $x \in \Cbb^\ast$,
\begin{align} %
\lim_{q \to 1 - 0}
\widetilde{\psi}_\la^{\mathrm{A}} (( 1 - q ) x )
=
\Gamma ( \beta ) x^{1 - \beta} e^x.
\end{align}
\end{thm}
\begin{proof} 
We put $x\mapsto (1-q)x$ and $b\mapsto q^\beta$ in Theorem \ref{qLBA} to obtain the limit. Then, we have
\begin{align} %
\widetilde{\psi}_\la^{\mathrm{A}} (( 1 - q ) x )
&=\frac{(q;q)_\infty}{(q^\beta ;q)_\infty}\frac{\theta_q(\lambda )}{\theta_q(q^{1-\beta }\lambda )}
\frac{\theta_q\left(\frac{q^{1-\beta }}{(1-q)}\frac{\lambda}{x}\right)}{\theta_q\left(\frac{1}{(1-q)}\frac{\lambda}{x}\right)}\sum_{n\ge 0}\frac{(1-q)^n}{(q;q)_n}x^n\notag\\
&=\frac{(q;q)_\infty}{(q^\beta ;q)_\infty}(1-q)^{1-\beta}\frac{\theta_q(\lambda )}{\theta_q(q^{1-\beta }\lambda )}
\frac{\theta_q\left(\frac{q^{1-\beta }}{(1-q)}\frac{\lambda}{x}\right)}{\theta_q\left(\frac{1}{(1-q)}\frac{\lambda}{x}\right)}(1-q)^{\beta -1}\sum_{n\ge 0}\frac{(1-q)^n}{(q;q)_n}x^n\notag\\
&=\Gamma_q(\beta )\frac{\theta_q(\lambda )}{\theta_q(q^{1-\beta }\lambda )}\frac{\theta_q\left(\frac{q^{1-\beta }}{(1-q)}\frac{\lambda}{x}\right)}{\theta_q\left(\frac{1}{(1-q)}\frac{\lambda}{x}\right)}(1-q)^{\beta -1}\sum_{n\ge 0}\frac{(1-q)^n}{(q;q)_n}x^n \label{lim1}
\end{align}
Now, applying \eqref{limt2} in Proposition \ref{limthetaq} to equation \eqref{lim1}, we have
\begin{align*} %
\lim_{q \to 0}
\widetilde{\psi}_\la^{\mathrm{A}} (( 1 - q ) x )
=
\Gamma (\beta )\times (\lambda )^{1-\beta }\times \left(\frac{\lambda}{x}\right)^{\beta -1}\sum_{n\ge 0}\frac{1}{n!}x^n
=
\Gamma (\beta )(x)^{1-\beta}e^x,
\end{align*}
which is the consistent limiting formula of Theorem \ref{qLBA} as $q \to 1 - 0$.
\end{proof}

\subsection{On the degeneration B}
Finally, we treat with Theorem \ref{qLBB} in the limit $q\to 1-0$.
With rescaling $x \mapsto x/( 1 - q )$, the limit $q \to 1 - 0$ of Theorem \ref{qLBB} turns out to be the following formula:
\begin{thm} 
For any $x \in \Cbb^\ast$,
\begin{align} %
\lim_{q \to 1- 0}
\widetilde{\psi}_\la^{\mathrm{B}} \left( \frac{x}{1 - q} \right)
=
\Gamma ( 1 - \alpha ) x^{- \alpha} e^{\frac{1}{x}}.
\end{align}
\end{thm}
\begin{proof} 
We put $x\mapsto x/(1-q)$ and $a \mapsto q^\alpha$ in Theorem \ref{qLBB} to consider the limit $q\to 1-0$. Then, we have 
\begin{align} %
\widetilde{\psi}_\la^{\text{B}} \left(\frac{x}{1-q}\right)
&=
\frac{(q;q)_\infty}{(q^{1-\alpha};q)_\infty}\frac{\theta_q(q^\alpha \lambda )}{\theta_q(\lambda )}\frac{\theta_q\left(\frac{q^{\alpha +1}}{(1-q)}\frac{x}{\lambda}\right)}{\theta_q\left(\frac{q}{(1-q)}\frac{x}{\lambda}\right)}
\sum_{n\ge 0}\frac{(1-q)^n}{(q;q)_n}\left(\frac{1}{q^\alpha x}\right)^n\notag\\
&=\frac{(q;q)_\infty}{(q^{1-\alpha};q)_\infty}(1-q)^{\alpha}\frac{\theta_q(q^\alpha \lambda )}{\theta_q(\lambda )}\frac{\theta_q\left(\frac{q^{\alpha +1}}{(1-q)}\frac{x}{\lambda}\right)}{\theta_q\left(\frac{q}{(1-q)}\frac{x}{\lambda}\right)}(1-q)^{-\alpha}
\sum_{n\ge 0}\frac{(1-q)^n}{(q;q)_n}\left(\frac{1}{q^\alpha x}\right)^n\notag\\  
&=\Gamma_q(1-\alpha )\frac{\theta_q(q^\alpha \lambda )}{\theta_q(\lambda )}\frac{\theta_q\left(\frac{q^{\alpha +1}}{(1-q)}\frac{x}{\lambda}\right)}{\theta_q\left(\frac{q}{(1-q)}\frac{x}{\lambda}\right)}(1-q)^{-\alpha}
\sum_{n\ge 0}\frac{(1-q)^n}{(q;q)_n}\left(\frac{1}{q^\alpha x}\right)^n. \label{lim2}
\end{align}
By the limiting formula \eqref{limt2} in Proposition \ref{limthetaq}, we can also obtain the well-defined limit $q\to 1-0$ of \eqref{lim2},
\begin{align*} %
\lim_{q \to 1 - 0}
\widetilde{\psi}_\la^{\text{B}} \left(\frac{x}{1-q}\right)
=
\Gamma (1-\alpha )\times (\lambda )^{-\alpha} \times \left(\frac{x}{\lambda}\right)^{-\alpha}\sum_{n\ge 0}\frac{1}{n!}\left(\frac{1}{x}\right)^n
=
\Gamma (1-\alpha )x^{-\alpha }e^{\frac{1}{x}}. 
\end{align*}
Therefore, we obtain the limit $q\to 1-0$ of our new formulae $\widetilde{\psi}_\la^{\mathrm{B}} ( x )$.
\end{proof}

\providecommand{\href}[2]{#2}\begingroup\raggedright
\endgroup

\end{document}